\documentclass[12pt]{amsart}
\usepackage{amsthm}
\usepackage{amsmath,amsxtra,amssymb,amsfonts,latexsym, mathrsfs, dsfont,bm,mathtools}
\usepackage[initials,alphabetic]{amsrefs}
\usepackage[usenames]{color}
\usepackage[all]{xy}
\usepackage{hyperref}

\allowdisplaybreaks[1]

\newtheorem{lemma}{Lemma}[section]
\newtheorem{prop}[lemma]{Proposition}
\newtheorem{theorem}[lemma]{Theorem}
\newtheorem{cor}[lemma]{Corollary}

\theoremstyle{remark}
\newtheorem{rem}[lemma]{Remark}

\theoremstyle{definition}
\newtheorem{defn}[lemma]{Definition}
\newtheorem{exam}[lemma]{Example}

\numberwithin{equation}{section}

\newcommand{\lge}{\langle}
\newcommand{\rge}{\rangle}

\newcommand{\fx}{\mathcal{F}}

\newcommand{\lx}{\mathcal{L}}

\newcommand{\nx}{\mathcal{N}}

\newcommand{\ux}{\mathcal{U}}

\newcommand{\nz}{\mathbb{N}}
\newcommand{\rz}{\mathbb{R}}
\newcommand{\cz}{\mathbb{C}}
\newcommand{\ez}{\mathbb{E}}
\newcommand{\fz}{\mathbb{F}}
\newcommand{\tz}{\mathbb{T}}
\newcommand{\zz}{\mathbb{Z}}
\newcommand{\pz}{\mathbb{P}}

\newcommand{\Ga}{\Gamma}
\newcommand{\Om}{\Omega}

\newcommand{\al}{\alpha}
\newcommand{\de}{\delta}
\newcommand{\si}{\sigma}
\newcommand{\ga}{\gamma}
\newcommand{\la}{\lambda}
\newcommand{\bt}{\beta}

\newcommand{\om}{\omega}

\newcommand{\eps}{\varepsilon}
\newcommand{\8}{\infty}
\newcommand{\td}{\widetilde}
\mathchardef\dash="2D

\newcommand{\Aut}{\operatorname{Aut}}
\newcommand{\sgn}{\operatorname{sgn}}
\newcommand{\Dom}{\operatorname{Dom}}
\newcommand{\Fix}{\operatorname{Fix}}
\newcommand{\Ent}{\operatorname{Ent}}
\newcommand{\dd}{\mathrm{d}}
\newcommand{\spec}{\operatorname{spec}}

\begin{document}
  \title{Subgaussian 1-cocycles on discrete groups}
\author{Marius Junge}
\date{\today}
\address{Department of Mathematics, University of Illinois, Urbana, IL 61801}
\email{mjunge@illinois.edu}
\author{Qiang Zeng}
\address{Department of Mathematics, University of Illinois, Urbana, IL 61801}
\curraddr{Center of Mathematical Sciences and Applications, Harvard University, Cambridge, MA 02138}
\email{qzeng@cmsa.fas.harvard.edu}
\thanks{The first author was partially supported by NSF Grant DMS-1201886}
\maketitle

\begin{abstract}
 We prove the $L_p$ Poincar\'e inequalities with constant $C\sqrt{p}$ for $1$-cocycles on countable discrete groups under Bakry--Emery's $\Ga_2$-criterion. These inequalities determine an analogue of subgaussian behavior for 1-cocycles. Our theorem improves some of our previous results in this direction, and in particular implies Efraim and Lust-Piquard's Poincar\'e type inequalities for the Walsh system. The key new ingredient in our proof is a decoupling argument. As complementary results, we also show that the spectral gap inequality implies the $L_p$ Poincar\'e inequalities with constant $C{p}$ under some conditions in the noncommutative setting. New examples which satisfy the $\Ga_2$-criterion are provided as well.
\end{abstract}

\section{Introduction}
Subgaussian behavior of random variables and stochastic processes is an important topic in probability theory. It is closely related to the concentration of measure phenomenon; see e.g. \cite{Ver}. Functional inequalities -- including log-Sobolev inequality, Poincar\'e inequality, transportation-entropy inequalities -- have played a critical role in the development of this theory in the last two decades; see \cite{Ta96, BG, OV, BGL} and the references therein for the extensive literature. More recently, this theory has been applied to study random matrices; see e.g. \cite{Ver, Gui}. In this paper, we want to connect this well-known theory in classical probability to 1-cocycles on groups, which is important in both group theory (Kazhdan's Property (T), the Haagerup property, etc.) and operator algebras; see e.g. \cite{BO}. We are interested in determining a class of 1-cocycles which satisfy an analogue of the subgaussian growth condition via Poincar\'e type inequalities. Recall that a random variable $X$ is subgaussian if and only if $(\ez|X|^p)^{1/p}\le C\sqrt{p}$ for all $p\ge 1$. Here and in the following we use $C, C',C_1$, etc.\! to denote absolute constants which may vary from line to line. To generalize this notion, we consider the following $L_p$ Poincar\'e inequalities for a probability space $(\rz^d,\mu)$,
\begin{equation}\label{pcrc}
\|f-\ez_\mu(f)\|_p \le C\sqrt{p}\|\nabla f\|_p
\end{equation}
for all $p\ge 2$ and differentiable $f\in L_\8(\rz^d, \mu)$. Observe that \eqref{pcrc} resembles subgaussian growth of random variables. In particular, choosing $f(x)=x$ when $d=1$ we recover the classical definition except for $1\le p<2$.

As a classical example, the Gaussian measure on $\rz^d$ satisfies \eqref{pcrc} due to Pisier \cite{Pis86}; see \cite{JZ12} for another proof. More classical examples satisfying \eqref{pcrc} can be found in \cite{AW13} and the references therein. In fact, one way to generalize \eqref{pcrc} is via the semigroup theory of operators. The analogue of gradient in this context is Meyer's ``carr\'e du champs''. We can even go further and consider an analogue of \eqref{pcrc} in a noncommutative $W^*$ probability space. Recall from \cite{VDN} that $(\nx,\tau)$ is a $W^*$ probability space if $\nx$ is a von Neumann algebra and $\tau$ is a normal state. We also assume that $\nx$ is finite and $\tau$ is tracial and faithful. Throughout we always work with a \emph{standard} semigroup $T_t=e^{-tA}$ acting on $(\nx,\tau)$ with generator $A$. Here a standard semigroup $(T_t)_{t\ge0}$ is pointwise $\si$-weak (weak$*$) continuous such that every $T_t$ is normal unital completely positive and symmetric on $L_2(\nx, \tau)$. We define the gradient form associated to $A$ (Meyer's ``carr\'e du champs'') as
\[
\Ga^A(f_1,f_2)=\frac12[A(f_1^*)f_2+f_1^*A(f_2)-A(f_1^*f_2)]
\]
for $f_1,f_2$ in a suitable involutive subalgebra of the domain of the generator, which is supposed to exist. In the following, we may simply write $\Ga$ for $\Ga^A$ if the generator under consideration is clear. Let ${\rm Fix}=\{x\in \nx:T_t x=x ,\forall t>0\}$ be the fixed point algebra of $T_t$. It was shown in \cite{JX07} that ${\rm Fix}$ is a von Neumann subalgebra of $\nx$. Thus there exists a unique conditional expectation $E_{\Fix}: \nx\to {\rm Fix}$. Recall that the noncommutative $L_p$ space $L_p(\nx,\tau)$ is defined as the closure of $\nx$ in the norm $\|\cdot\|_p$ given by $\|x\|_p=[\tau((x^*x)^{p/2})]^{1/p}$ for $0<p<\8$ and $\|x\|_\8=\|x\|$ for $p=\8$, where $\|\cdot\|$ is the operator norm. We usually write $L_p(\nx)$ for short. It is well known that $L_p(\nx)$ is a Banach space for $1\le p\le \8$; see \cite{PX03} for more details.
\begin{defn}\label{def1}
 A standard semigroup $T_t$ acting on $(\nx,\tau)$ is said to be subgaussian if the following $L_p$ Poincar\'e inequalities
\begin{equation}\label{pcr2}
 \|f- E_{\Fix}(f)\|_p \le C\sqrt{p}\max\{\|\Ga^A(f,f)^{1/2}\|_p,~ \|\Ga^A(f^*,f^*)^{1/2}\|_p\}.
\end{equation}
hold for $2\le p<\8$ and $f$ in a suitable involutive subalgebra of the domain of the generator.
\end{defn}
For simplicity, in the following we may say the above inequality holds for all $f\in \nx$, since it is automatically true if the right-hand side is infinity. Since the gradient form $\Ga$ coincides with the modulus of the gradient if $-A$ is the Laplacian of a Euclidean space, \eqref{pcr2} is indeed a generalization of \eqref{pcrc}. It is known that for classical diffusion semigroups, log-Sobolev inequality implies \eqref{pcr2}; see \cite{AS94} and also \cite{AW13}. Efraim and Lust-Piquard proved that \eqref{pcr2} holds for Walsh systems and CAR algebras in \cite{ELP}. In fact, we started to study the subgaussian behavior \eqref{pcr2} of semigroups acting on a general noncommutative $W^*$ probability space $(\nx,\tau)$ in \cite{JZ12}. It was shown in \cite{Z13} that the group measure space $L_\8(\rz^d,\ga_d)\rtimes G$ satisfies \eqref{pcr2}, where the action and the gaussian measure $\ga_d$ are associated to an orthogonal representation of $G$ on a real Hilbert space, and the semigroup acting on $L_\8(\rz^d,\ga_d)\rtimes G$ is a natural extension of the Ornstein-Uhlenbeck semigroup on $L_\8(\rz^d,\ga_d)$. A remarkable consequence of \eqref{pcr2} is that one can get concentration inequalities, exponential integrability and transportation cost inequalities; see \cite{ELP, Z13}. Our goal here is to prove \eqref{pcr2} for group von Neumann algebras under some conditions on the 1-cocycles of groups and to elaborate on the relationship between the spectral gap of $A$ and $L_p$ Poincar\'e inequalities for semigroups acting on a $W^*$ probability space $(\nx,\tau)$.

Let us be more precise. Let $G$ be a countable discrete group.  Recall that a (generic) conditionally negative length (or cn-length for short) function $\psi$ on $G$ determines a 1-cocycle $b_\psi$ on $G$ with coefficients in an orthogonal representation $(\al, H_\psi)$ of $G$, and vice versa. Let $\la: G\to B(\ell_2(G))$ be the left regular representation given by $\la_g(\de_h)=\de_{gh}$ for $g,h\in G$, where $\de_h$'s form a unit vector basis of $\ell_2(G)$. The group von Neumann algebra $LG$ is the closure of linear span of $\la(G)$ in the weak operator topology. It is well known that $LG$ admits a canonical normal faithful tracial state given by $\tau(f)=\lge \de_e,f\de_e\rge $ for $f\in LG$, where $e$ is the identity element of $G$. Consider the semigroup $T_t$ acting on $LG$ defined by $T_t\la(g)=e^{-t\psi(g)}\la(g)$ for $g\in G$. Then $(T_t)_{t\ge0}$ is a standard semigroup on $(LG,\tau)$. Thus $(T_t)_{t\ge0}$ extends to a strongly continuous semigroup of contractions on $L_2(LG)\cong \ell_2(G)$ and the generator is given by $A\la(g)=\psi(g)\la(g)$.
We say that a 1-cocycle $b_\psi$ on $G$ with coefficients in the orthogonal representation $\al: G\to H_\psi$ is \emph{subgaussian} if the semigroup given by $T_t\la(g)=e^{-t\psi(g)}\la(g)$ is subgaussian in the sense of Definition \ref{def1}, i.e.,
 \begin{equation}\label{pcr21}
 \|f- E_{\Fix}(f)\|_p \le C\sqrt{p}\max\{\|\Ga^\psi(f,f)^{1/2}\|_p,~ \|\Ga^\psi(f^*,f^*)^{1/2}\|_p\}.
 \end{equation}
holds for all $f\in LG$ and $2\le p<\8$.

 For readers who are not familiar with von Neumann algebras, \eqref{pcr21} can be formulated in a more algebraic way, i.e.,
 \[
 \|f\|_p \le C\sqrt{p}\max\{\|\Ga^\psi(f,f)^{1/2}\|_p, ~ \|\Ga^\psi(f^*,f^*)^{1/2}\|_p\}
 \]
 for $f\in \cz G$, $f=\sum_{s\in G, \psi(s)\neq 0} a_s s$, where $\cz G$ is the group algebra of $G$ (thus $f$ is a finite linear combination),
\[
\Ga^\psi(f,f) =\sum_{s,t\in G}\bar{a}_s a_t K(s,t)s^{-1}t
\]
and $K(s,t)$ is the Gromov form given by
\[
K(s,t) =\frac12(\psi(s)+\psi(t)-\psi(s^{-1}t)), \quad s, t\in G.
\]
Here $\psi(s)\neq 0$ in the linear combination implies that $E_{\Fix} f =0$.  We remark that  $\Ga^\psi(f,f)$ and $\Ga^\psi(f^*,f^*)$ are not equal in general because of noncommutativity. It is clear that in this formulation \eqref{pcr21} is really a condition on the 1-cocycle (or the cn-length function) and involves no probability theory or semigroups of operators. However, the only way we know to prove such inequalities is to use probability in an efficient way. To state our main results, we need to introduce the well-known $\Ga_2$-criterion due to Bakry--Emery. Recall that
\[
\Ga_2^A(f_1,f_2)=\frac12[\Ga^A(Af_1,f_2)+\Ga^A(f_1,Af_2)-A\Ga^A(f_1,f_2)]
\]
whenever $f_1$ and $f_2$ are in a suitable involutive subalgebra of the domain of the generator.
\begin{theorem}\label{mthe}
  Let $G$ be a countable discrete group with cn-length function $\psi$ and $LG$ its group von Neumann algebra. Suppose $f\in LG$ satisfies $\Ga_2^\psi(f,f) \ge \al \Ga^\psi(f,f)$ for some $\al>0$. Then for $2\le p<\8$,
  \begin{equation}\label{pcr3}
      \|f-E_{\Fix} f\|_p\le C \sqrt{p/\al} \max\{\|\Ga^\psi(f,f)^{1/2}\|_p, \|\Ga^\psi(f^*,f^*)^{1/2}\|_p\}.
  \end{equation}
\end{theorem}
Note that Theorem \ref{mthe} is a result for individual elements. As a property of the group von Neumann algebra, we hope to show that \eqref{pcr3} holds for all $f\in LG$. Recall Bakry--Emery's $\Ga_2$-criterion \cite{BE85}: There exists $\al>0$ such that $\Ga_2^\psi(f,f)\ge \al \Ga^\psi(f,f)$ for all $f\in LG$ for which both $\Ga_2^\psi(f,f)$ and $\Ga^\psi(f,f)$ are well-defined. As observed in \cite{JZ12}, in our context this condition is equivalent to the algebraic condition that $\Ga_2^\psi-\al \Ga^\psi$ is a positive semidefinite form. The domain of $\Ga_2$ is typically smaller than that of $\Ga$. Therefore, it is possible that $\Ga_2(f,f)$ is not well defined for some element $f\in LG$ while \eqref{pcr3} still holds for this element. However, by \cite{JZ12}*{Corollary 4.8}, we know that \eqref{pcr3} always holds for all $f\in LG$ provided $\Ga_2(f,f)\ge \al\Ga(f,f)$ for all $f\in\cz G$, because $\cz G$ is a weakly dense subalgebra of $LG$. Let us record this as the following result.
\begin{cor}\label{sgau}
Suppose the $\Ga_2$-criterion holds for the cn-length function $\psi$ on a group $G$. Then we have the $L_p$ Poincar\'e inequalities \eqref{pcr3} for all $f\in LG$ and $2\le p<\8$ whenever the right-hand side of \eqref{pcr3} is finite. Therefore, the 1-cocycle $b_\psi$ is subgaussian.
\end{cor}

Our motivation to study this problem comes from both noncommutative harmonic analysis and probability theory. In noncommutative harmonic analysis, Poincar\'e inequalities are closely related to noncommutative Riesz transform and smooth Fourier multiplier theory developed in \cite{JM10,JMP}. In probability theory, precise moment estimation of random variables could be the starting point of various results, including concentration and transportation inequalities.

Let us mention some interesting applications. As indicated in \cite{JZ12}, applying Theorem \ref{mthe} to the group $G=\zz_2\times\cdots\times \zz_2$, we recover the Poincar\'e type inequalities for the Walsh system due to Efraim and Lust-Piquard \cite{ELP}. By embedding the matrix algebra into the discrete Heisenberg group von Neumann algebra, we find subgaussian behavior for matrix algebras. Another immediate consequence of our main results is the following transportation type inequalities shown in \cite{JZ12,Z13}. Let us recall some notation. Let $T_t=e^{-tA}$ be a semigroup acting on a noncommutative probability space $(\nx,\tau)$ with generator $A$. Given $\tau$-measurable operators $\rho$ and $\si$, we define the following analogues of classical Wasserstein distances
\[
 Q_1(\rho,\si):=\sup\{|\tau(x\rho)/\tau(\rho) -\tau(x\si)/\tau(\si)|: x \text{ self-adjoint, } \|\Ga(x,x)\|_\8\le 1\}
\]
and
\[
 Q_{\phi}(\rho,\si)=\sup \{|\tau(x\rho)/\tau(\rho) -\tau(x\si)/\tau(\si)|: x \mbox{ self-adjoint, } \|\Ga(x,x)^{1/2}\|_{\phi}\le 1\}
\]
where $\|y\|_\phi=  \inf \{c>0: \tau[\phi(|y|/c)]\le 1\}$ and $\phi(t)=e^{t^2}-1$; see \cite{Z13} for a detailed discussion about these distances and their relationship to Rieffel's quantum metric spaces. For a $\tau$-measurable positive operator $\rho$, we define the entropy
\[
 \Ent(\rho)=\tau[\rho\ln(\rho/\tau(\rho))].
\]
\begin{cor}\label{tran}
  Suppose the $\Ga_2$-criterion holds for the cn-length function $\psi$ on a discrete group $G$. Then
 \[
Q_1(\rho,E_{\Fix}\rho)\le C\sqrt{\Ent(\rho)}
 \]
and
\[
 Q_{\phi}(\rho, E_{\Fix}\rho)\le C'\max\{\sqrt{\Ent(\rho)}, ~ \Ent(\rho)\}
\]
for all $\tau$-measurable positive operators $\rho$ affiliated to $LG$ with $\tau(\rho)=1$.
\end{cor}
We remark that the constant of order $\sqrt{p}$ in our Poincar\'e inequalities is crucial to deduce these entropy bounds as observed in \cite{JZ12, Z13}. A constant of the order $p$, as obtained in Section 4, is not sufficient for such entropy bounds.

Let us now point out the connection of our results to some previous ones. As is well known, the major application of $\Ga_2$-criterion is to derive Gross' log-Sobolev inequality (LSI) under some mild condition; see \cite{BE85} and also the lecture notes \cite{GZ} for more details in this direction. However, as observed in \cite{JZ12}, this implication is not true in general non-diffusion setting where the sample paths are discontinuous; see e.g. \cite{BGL14, RY} for the definition of classical diffusion semigroups and processes. In particular, the diffusion property is characterized by the Leibnitz rule on the $\Ga$ form
\begin{equation}\label{leib}
  \Ga(fg, h) = f\Ga(g,h)+g\Ga(f,h)
\end{equation}
for smooth functions $f,g$ and $h$ in the domain of the generator. On the other hand, we can deduce concentration inequalities directly from $\Ga_2$-criterion without \eqref{leib}; see \eqref{wsga} below. Of course we still need a certain regularity condition on the semigroup $T_t$:
\begin{equation}\label{ncdif}
\Ga^A(x,x)\in L_1(\nx)\text{ for all } x\in \Dom(A)\cap\nx
\end{equation}
(more precisely, for all $x\in \Dom(A^{1/2})\cap \nx$ by extension).

{This condition is introduced in \cite{JRS} to characterize the semigroups which admit a Markov dilation with certain nice properties in analogy to classical diffusion processes.} For example the Poisson semigroup on the circle satisfies \eqref{ncdif}, but it is not a classical diffusion semigroup. For a standard semigroup $T_t$ with \eqref{ncdif}, it was proved in \cite{JZ12} that the $\Ga_2$-criterion implies the following Poincar\'e type inequalities
\begin{equation}\label{wsga}
  \|f-E_{{\rm Fix}}(f)\|_p \le C {\al^{-1/2}}\min\{\sqrt{{p}} \|\Gamma^A(f,f)^{1/2}\|_\8, ~p \|\Gamma^A(f,f)^{1/2}\|_p \}
\end{equation}
for all self-adjoint $f\in \nx$. The obstruction of inequalities like \eqref{pcr2} in the noncommutative setting was a lack of the good Burkholder inequality with appropriate norms or constants. Indeed, with the help of the optimal Burkholder--Davis--Gundy (BDG) inequality, it was proved that the classical diffusion semigroups satisfy \eqref{pcr2} under the $\Ga_2$-criterion; see \cite[Theorem 4.9]{JZ12}. This may be regarded as a shortcut of the following implication in the classical diffusion setting
\begin{equation}\label{impl}
 \Ga_2\dash\text{criterion}\Rightarrow \text{ log-Sobolev inequality} \Rightarrow \eqref{pcr2}.
\end{equation}
Here the first implication was due to Bakry--Emery \cite{BE85} and the second was due to Aida--Stroock \cite{AS94}.

The optimal classical BDG inequality due to Barlow and Yor \cite{BY82} asserts that
\[
\|X_t \|_p\le C\sqrt{p}\|\lge X, X\rge_t^{1/2}\|_p
\]
for any continuous mean 0 martingale $X$, where $\lge X, X\rge$ is the quadratic variation of $X$. One way to obtain such an inequality in the noncommutative setting is through the Burkholder inequality \cite{JX03}
 \begin{equation}\label{burk}
  \Big\|\sum_k dx_k\Big\|_p
 \le  A(p)\Big\|\Big(\sum_k  E_{k-1} (dx_k^*dx_k + dx_kdx_k^*)\Big)^{1/2}\Big\|_p + B(p)\Big(\sum_k \|dx_k\|_p^p\Big)^{\frac1p}
\end{equation}
where $dx_k= x_k- x_{k-1}$ is the martingale difference associated to the martingale $(x_k,\nx_k)$ and $E_k: \nx\to \nx_k$ is the conditional expectation. One would expect the best order of $A(p)$ is $\sqrt{p}$, which is indeed the case in the commutative theory \cite{Pin}. The difficulty in the noncommutative generality can be seen from the fact that if one requires $A(p)=B(p)=c(p)$, then the optimal order of $c(p)$ in \eqref{burk} is known to be $p$ \cite{Ran, JX05}, compared to $p/\ln p$ in the commutative theory \cite{Hi90}. This shows that general noncommutative martingales exhibit quite different behaviors from the classical martingales so that $A(p)=O(\sqrt{p})$ may not be true. Although it is still unclear to us whether $A(p)$ can be reduced to $\sqrt{p}$ in the general noncommutative setting, we do resolve an important case of this problem in this paper, which is good enough to establish Theorem \ref{mthe}. In this way we improve the main results of \cite{JZ12} for the case of semigroups acting on group von Neumann algebras generated by 1-cocycles. Our proof follows the same strategy as that in \cite{JZ12}. The difficulty mentioned above is overcome by a decoupling argument, which is the key new ingredient (Lemma \ref{deco}) in our proof. We refer the interested reader to the monograph \cite{dG} for various aspects of decoupling and applications.

Let us conclude the introduction by mentioning the relationship among log-Sobolev inequality, spectral gap inequality and $L_p$ Poincar\'e inequalities. It is well known that the log-Sobolev inequality implies the existence of spectral gap, or equivalently, $L_2$ Poincar\'e inequality. Conversely, the spectral gap inequality together with a {defective} log-Sobolev inequality yields the log-Sobolev inequality; see e.g. \cite{GZ} for these facts. On the other hand, the $L_p$ Poincar\'e inequalities obviously imply the spectral gap inequality. It would be interesting to determine when the converse implication is possible. It is known that in the classical diffusion setting the spectral gap would imply $L_p$ Poincar\'e inequalities, but with constant $Cp$; see e.g. \cite[Proposition 2.5]{Mi09}. We show similar results in the noncommutative setting under some conditions. In Section 4, we formulate certain results in this direction and prove, e.g., the following:
\begin{theorem}\label{spg1}
 Let $T_t$ be an ergodic standard semigroup acting on a diffuse probability space $(\nx,\tau)$ which satisfies \eqref{compa}. Suppose the generator $A$ of $T_t$ has a spectral gap: For $f\in \nx$,
$$\|f-\tau(f)\|_2\le C\max\{\|\Ga(f,f)^{1/2}\|_2, \|\Ga(f^*,f^*)^{1/2}\|_2\}.$$
Then for all even integer $p\ge 2$ and all $f\in \nx$,
\[
 \|f-\tau(f)\|_p\le C'p \max\{\|\Ga(f,f)^{1/2}\|_p, \|\Ga(f^*,f^*)^{1/2}\|_p\}.
\]
\end{theorem}
{We overcome the lack of \eqref{leib}} in the noncommutative setting by using the derivations of noncommutative Dirichlet forms developed by Cipriani and Sauvageot \cite{JRS} and the $L_p$ regularity theorem due to Olkiewicz--Zegarlinski \cite{OZ99}.

The paper is organized as follows. We recall some preliminary facts in Section 2. Then we prove the $L_p$ Poincar\'e inequalities with constant $C\sqrt{p}$ in Section 3. The relationship between the spectral gap inequality and $L_p$ Poincar\'e inequalities is discussed in Section 4. Some examples and illustrations are given in Section 5.

\section{Preliminaries}
\subsection{Crossed products}\label{crosp}
We briefly recall the crossed product construction. Our reference is  \cite{Tak, JMP}. Let $G$ be a discrete group with left regular representation $\la: G\to B(\ell_2(G))$. Given a noncommutative probability space $(\nx,\tau)$, we may assume $\nx\subset B(H)$ for some Hilbert space $H$. Suppose a trace preserving action $\al$ of $G$ on $\nx$ is given, i.e., we have a group homomorphism $\al: G\to \Aut(\nx)$ (the $*$-automorphism groups of $\nx$) with $\tau(x)=\tau(\al_g(x))$ for all $x\in\nx, g\in G$. Identify $\ell_2(G)\otimes H$ with $\ell_2(G; H)$. Consider the representation $\pi$ of $\nx$ on $\ell_2(G;H)$ given by $\pi(x)=\sum_{g\in G}\al_{g^{-1}}(x)\otimes e_{g,g}$, where $e_{g,h}$ is the matrix unit of $B(\ell_2(G))$. In other words, $\pi(x)\xi(g)=\al_{g^{-1}}(x)\xi(g)$ for $x\in \nx, \xi\in \ell_2(G;H)$. Then the crossed product of $\nx$ by $G$, denoted by $\nx\rtimes_\al G$, is defined as the weak operator closure of $1_\nx\otimes \la(G)$ and $\pi(\nx)$ in $B(\ell_2(G;H))$. We usually drop the subscript $\al$ if there is no ambiguity. Clearly, $\nx\rtimes G$ is a von Neumann subalgebra of $\nx\overline \otimes B(\ell_2(G))$. In the special case $\nx=\cz$, the complex number algebra, $\cz\rtimes G$ reduces to the group von Neumann algebra $\lx(G)$. Therefore, $\lx(G)$ is a von Neumann subalgebra of $\nx\rtimes G$ and there exists a unique conditional expectation $E_{\lx(G)}:\nx\rtimes G\to \lx(G)$. If $g\in G, f_g\in \nx$, we simply write $f_g\rtimes \la(g)$ or even $f_g\la(g)$ for $\pi(f_g)(1_\nx\otimes \la(g))$. A generic element of $\nx\rtimes G$ can be written as
\begin{align*}
  \sum_{g\in G}f_g\rtimes \la(g)&=\sum_{g\in G}\pi(f_g)\la(g) = \sum_{g,h,h'} (\al_{h^{-1}}(f_g)\otimes e_{h,h})(1_\nx\otimes e_{gh',h'})\\
  &=\sum_{g,h}\al_{h^{-1}}(f_g)\otimes e_{h,g^{-1}h}.
\end{align*}
There is a canonical trace on $\nx\rtimes G$ given by $$\tau\rtimes\tau_G(f\rtimes \la(g))=\tau\otimes \tau_G(f\otimes\la(g))=\tau(f)\de_{g=e},$$
where we denote by $\tau_G$ the canonical trace on $\lx(G)$. The arithmetic in $\nx\rtimes G$ is given by
$(f\rtimes \la(g))^*=\al_{g^{-1}}(f^*)\rtimes \la(g^{-1})$ and $(f_1\rtimes \la(g_1))(f_2\rtimes \la(g_2))=(f_1\al_{g_1}(f_2))\rtimes \la(g_1g_2)$. In what follows, we may simply write $f\la(g)$ instead of $f\rtimes \la(g)$.

\subsection{1-cocycles on groups}\label{1coc}
Let $G$ be a countable discrete group with a conditionally negative length (cn-length) function $\psi: G\to\rz_+$. Recall that $\psi$ is a length function if $\psi(e)=0$ and $\psi(g)=\psi(g^{-1})$, and $\psi$ is conditionally negative if $\sum_g a_g=0\Rightarrow \sum_{g,h}\bar a_g a_h\psi(g^{-1}h)\le 0$. Then $\psi$ determines an affine representation which is given by an orthogonal representation $\al:G\to O(H_\psi)$ over a real Hilbert space $H_\psi$ together with a map $b_\psi:G\to H_\psi$ satisfying the cocycle law, i.e., $b_\psi(gh)=b_\psi(g)+\al_g(b_\psi(h))$; see e.g. \cite{BO}. To be more concrete, let $\rz G$ be the algebraic group algebra of $G$. Put $K(g,h)=\frac12(\psi(g)+\psi(h)-\psi(g^{-1}h))$ for $g,h\in G$ and define
\[
 [\sum_g a_g \de_g, \sum_{g'}a_{g'}\de_{g'}]=\sum_{g,g'}a_g \bar a_{g'} K(g,g').
\]
Then $H_\psi$ is the closure of the quotient of $\rz G$ by the kernel of $[\cdot, \cdot]$, i.e., $H_\psi=\overline{\rz G/N_\psi}$ where $N_\psi$ is the kernel of $[\cdot, \cdot]$. Define $b_\psi(g)=\de_g+N_\psi$ for $g\in G$ and $\al_g(b_\psi(h))=b_\psi(gh)-b_\psi(g)$. In this way, we obtain a 1-cocycle $b_\psi$. Conversely, suppose that $b: G\to H$ is a 1-cocycle with coefficients in an orthogonal representation $(\al,H)$ of $G$. Put $\psi(g)=\|b(g)\|^2$ for $g\in G$. Then $\psi$ is a cn-length function on $G$. By a Gram--Schmidt procedure, we may choose an orthonormal basis of $H_\psi$ so that $b_\psi(g)$ depends on only finitely many nonzero coordinates for all $g\in G$. This observation will save us from some technical problems. We write $H_\psi\cong \ell_2(d)$ even if $d=+\8$.

\subsection{Gaussian measure space construction}\label{gaus}
Note that the Hilbert space $L_2([0,\8)) \otimes H_\psi$ is separable. By the well known Gaussian space construction (see e.g. \cite{RY,Str}), there exists a probability space $(\Om, \fx, \pz)$ and a linear map
\[
 \bt: L_2([0,\8)) \otimes H_\psi \to L_2(\Om, \pz)
\]
such that $\bt(1_{[0,t]}\otimes \xi)$ is Gaussian centered and
\[
\ez[\bt(1_{[0,t]}\otimes \xi) \bt(1_{[0,s]}\otimes \eta)]=2\lge 1_{[0,t]}\otimes \xi,  1_{[0,s]}\otimes \eta\rge_{L_2([0,\8)) \otimes H_\psi}=2\min\{t, s\}\lge \xi,\eta\rge_{H_\psi}.
\]
We simply write $\bt_t(\xi)=\bt(1_{[0,t]}\otimes \xi)$ and denote by $\fx_t$ the $\si$-subalgebra of $\fx$ generated by $\bt_s(\xi)$, for all $s\le t$ and $\xi\in H_\psi$. By Kolmogorov's continuity criterion (see, e.g., \cite[Theorem I.2.1]{RY}), $\bt_t(\xi)$ thus constructed is a $\rz^d$-valued Brownian motion, where $\rz^d$ is viewed as an abstract Wiener space associated to $H_\psi$ if $d=\8$. Indeed, by construction the $k$-th component of $\bt_t(\xi)$ is a 1-dimensional Brownian motion with mean 0 and variance $2t|\xi_k|^2$, where $\xi_k$ is the $k$-th component of $\xi$, and all the components of $\bt_t(\xi)$ are independent. More explicitly, we can simply take $\Om=(\rz^d)^{[0,+\8)}$. Then $\bt_t(\xi)(\om) =\sqrt{2}\sum_{k=1}^d \xi_k\om_t^k$, where $\om_t^k$ is the $k$-th coordinate map at time $t$. It is readily seen that $\bt_t(\xi)$ is a random variable in $(\Om,\pz)$ with variance $2t\|\xi\|^2$. Suppose $\al$ is an orthogonal representation of $G$ on $H_\psi$. By \cite[Theorem 8.3.14]{Str}, $\al$ determines a Gaussian measure preserving action $\al^*$ on $(\Om,\pz)$. By abuse of notation, we still denote $\al^*$ by $\al$. The $G$-action $\al$ on $\rz^d$ induces an action $\hat{\al}$ on $L_2(\Om, \pz)$, such that $\hat{\al}_g(\bt_t(h))=\bt_t(\al_g(h))$. It follows that 
\begin{equation}\label{eqv0}
\hat{\al}_g(f)(\om) = f(\al_{g^{-1}}(\om) )
\end{equation}
for $f\in L_2(\Om,\pz)$, where $\al_{g}(\om)_t=\al_g(\om_t)$. Clearly, $\hat\al$ extends naturally to isometric actions on $L_p(\Om,\pz)$ for $1\le p\le \8$. In the following we will consider the von Neumann algebra $L_\8(\Om, \pz)\rtimes_{\hat\al} G$ and  simply omit the subscript $\hat\al$ in the notation. To conclude this section, we remark that although $H_\psi$ (and thus $\bt_t(\xi)$) may be infinitely dimensional, $\bt_t(b_\psi(g))$ is always a finite dimensional Brownian motion for all $g\in G$ because $b_\psi(g)$ only depends on finitely many nonzero coordinates.

\subsection{Hardy spaces associated to martingales}\label{hard}
We refer to \cite{JM10, JP} for this subsection. Let $\nx_1\subset \nx_2\subset \cdots \subset \nx$ be a filtration with conditional expectations $E_n: \nx\to \nx_n$. Recall that a sequence  $(x_n)\subset \nx$ is a martingale if $x_n\in \nx_n$ and $E_n(x_{n+1})=x_n$. Let $dx_k=x_k-x_{k-1}$ be the associated martingale differences. We need the conditional Hardy spaces associated to martingales given as follows. For $1\le p<\8$, define
\[\|x\|_{h_p^d}=\Big(\sum_{k}\|dx_k\|_p^p\Big)^{1/p},  \|x\|_{h_p^c} = \left\|\Big(\sum_{k}E_{k-1}(dx_k^* dx_k)\Big)^{1/2}\right\|_p,\]
and $\|x\|_{h_p^r}=\|x^*\|_{h_p^c}$.

We are going to use the continuous filtration $(\nx_t)_{t\ge0}\subset \nx$ in the following. Recall that a martingale $x$ is said to have almost uniform (or a.u. \hspace{-0.3em}for short) continuous path if for every $T>0$, every $\eps>0$ there exists a projection $e$ with $\tau(1-e)<\eps$ such that the function $f_e:[0,T]\to \nx$ given by $f_e(t)=x_te\in \nx$ is norm continuous. Let $\si=\{0=s_0, \cdots ,s_n=T\}$ be a partition of the interval $[0,T]$ and $|\si|$ its cardinality. Put
\[
\|x\|_{h_p^c([0,T];\si)} = \Big\|\sum_{j=0}^{|\si|-1}E_{s_{j}}|E_{s_{j+1}}x-E_{s_j}x|^2 \Big\|^{1/2}_{p/2}, \quad 2\le p\le \8 ,
\]
\[
 \|x\|_{h_p^d([0,T];\si)} = \Big(\sum_{j=0}^{|\si|-1}\|E_{s_{j+1}}x-E_{s_j}x\|_p^p \Big)^{1/p}, \quad 2\le p<\8,
\]
and $\|x\|_{h_p^r([0,T];\si)}=\|x^*\|_{h_p^c([0,T];\si)}$. Let $\ux$ be an ultrafilter refining the natural order given by inclusion on the set of all partitions of $[0,T]$. Let $x\in L_p(\nx)$. For $2\le p< \8$, we define
\[\langle x,x\rangle_T= \lim_{\si, \ux}\sum_{i=0}^{|\si|-1}E_{s_i}|E_{s_{i+1}}x-E_{s_i}x|^2.\]
Here the limit is taken in the weak* topology and it is shown in \cite{JKPX} that the convergence is also true in  $L_{p}$ norm $\|\cdot\|_{p/2}$ for all $2<p<\8$. We define the continuous version of $h_p$ norms for $2\le p<\8$,
\[\|x\|_{h_p^c([0,T])}=\lim_{\si, \ux}\|x\|_{h_p^c([0,T];\si)},\]
\[\|x\|_{h_p^d([0,T])}=\lim_{\si, \ux}\|x\|_{h_p^d([0,T];\si)}.\]
and $\|x\|_{h_p^r([0,T])}=\|x^*\|_{h_p^c([0,T])}$ for $2\le p < \8$. Then for all $2<p<\8$
\begin{equation*}
 \|x\|_{h_p^c([0,T])}= \|\lge x,x\rge_T\|^{1/2}_{p/2}.
\end{equation*}
A martingale $x$ is said to be of vanishing variation if $\|x\|_{h^d_p([0, T])}=0$ for all $T>0$ and all $2<p<\infty$. If $x$ has a.u. \hspace{-0.3em}continuous path, then it is of vanishing variation. In the following, we will apply these results to matrix-valued martingales driven by Brownian motions. Hence they automatically have almost uniform continuous paths.

\section{$L_p$ Poincar\'e inequalities for group von Neumann algebras}
Consider the semigroup $T_t$ acting on $LG$ given by $T_t \la(g)=e^{-t\psi(g)} \la(g)$, where $\psi$ is a conditionally negative length function on $G$; see Section \ref{1coc}. $(T_t)$ satisfies \eqref{ncdif}. For a proof of this fact, see \cite{JZ12}. According to \cite{JRS}, $T_t$ admits a Markov dilation with almost uniformly continuous path. We refer the reader to \cite{JRS,JZ12} for the precise definition. In fact, we can write down the dilation explicitly in our setting. Following the notation of Section \ref{crosp} and \ref{gaus}, we define
\[
\pi_t: LG\to L_\8(\Om,\fx_t)\rtimes G, \quad \pi_t(
\la(g))=e^{i\bt_t(b_\psi(g))(\om)}\rtimes \la(g).
\]
The Markov property can be checked directly because
\[
E_s(\pi_t \la(g))=\pi_s(T_{t-s}\la(g))
\]
for $s<t$ and $g\in G$, and the same equation holds for arbitrary $f\in LG$ by linearity and density. Here $E_s: L_\8(\Om,\fx)\rtimes G\to L_\8(\Om,\fx_t)\rtimes G$ is the conditional expectation. It follows that
\begin{equation}\label{mart}
m_t(x)=\pi_t(x)-\pi_0(x)+\int_0^t \pi_s(A x) \dd s
\end{equation}
is a martingale with almost uniformly continuous path for $x\in  LG$. We will need the reversed martingale. To this end, let us fix a large constant $L>0$, and define
$$v_t(x)=\pi_t T_{L-t}x$$
for $t<L$. It is easy to check that $(v_t)_{0\le t\le L}$ is a martingale.

For $\xi\in H_\psi$ with finitely many nonzero coordinates, we write $\bt_t(\xi)=\sum_{k=1}^\8 \xi_k B_t^k$, where $(B_t^k)_k$ are independent Brownian motions with variance $2t$, and can be given by $B_t^k(\om)=\sqrt{2}\om_t^k$ in the notation of Section \ref{gaus}. By Ito's formula,
\begin{equation*}
   e^{i\bt_t(\xi)}=1+ i\sum_{k} \xi_k \int_0^t e^{i\bt_s(\xi)} \dd B_s^k - \|\xi\|^2\int_0^t e^{i\bt_s(\xi)} \dd s.
\end{equation*}
It follows that
\begin{equation}\label{ito}
  \pi_t(\la(g))=\la(g)+i\sum_{k} b_\psi(g)_k \Big(\int_0^t e^{i\bt_s(b_\psi(g))} \dd B_s^k\Big) \rtimes \la(g)-\int_0^t \pi_s(A(\la(g))) \dd s,
\end{equation}
where $b_\psi(g)_k$ is the $k$-th coordinate of $b_\psi(g)$. Combining \eqref{mart} and \eqref{ito}, we have
\[
m_t(\la(g)) = i\sum_{k} b_\psi(g)_k \Big(\int_0^t e^{i\bt_s(b_\psi(g))} \dd B_s^k\Big) \rtimes \la(g).
\]
Note that $v_t(\la(g))= e^{-(L-t)\psi(g)}e^{i\bt_t(b_\psi(g))}\la(g)$. By Ito's formula, we have
\[
 v_t(\la(g))=v_0(\la(g))+ i\sum_k b_\psi(g)_k\Big(\int_0^t e^{-(L-s)\psi(g)} e^{i\bt_s(b_\psi(g))}\dd B_s^k\Big)\rtimes \la(g).
\]
It follows that
\begin{align*}
   &\pi_L(\la(g))-\pi_0(T_L\la(g)) = \int_0^L \dd v_s(\la(g))\\
   &= i\sum_k b_\psi(g)_k\Big(\int_0^L e^{-(L-s)\psi(g)} e^{i\bt_s(b_\psi(g))}\dd B_s^k\Big)\rtimes \la(g).
\end{align*}
Let $x=\sum_{g\in G} x_g \la(g)\in \cz G$ be a finite sum. Then
\begin{equation}\label{rmart}
  \pi_L(x)- T_L(x)=i\sum_{g\in G, k\in \nz}x_g b_\psi(g)_k\Big(\int_0^L e^{-(L-s)\psi(g)} e^{i\bt_s(b_\psi(g))}\dd B_s^k\Big)\rtimes \la(g).
\end{equation}
We consider the discretized stochastic integral (assuming $n=L$), or martingale transform
\begin{equation}\label{dmar}
   M_n(x)=i\sum_{g\in G, j\in \nz} \sum_{k=0}^{n-1}x_g b_\psi(g)_j e^{-(L-t_k)\psi(g)}[ e^{i \bt_{t_k}(b_\psi(g))} d B_{t_k}^j ]\rtimes \la(g),
\end{equation}
where $d B_{t_k}^j= B_{t_{k+1}}^j-B_{t_k}^j$. {It is well known that this martingale converges to the stochastic integral in $L_p$ for $2\le p<\8$.} Indeed, the stochastic integral can be defined as the $L_2$ limit of a certain martingale transform; see e.g. \cite{KS91}. Similar argument can be applied to the case of $p>2$.
We need a precise Burkholder inequality for this (noncommutative) martingale in order to derive the subgaussian property. As explained in Introduction (see also \cite{JZ12}), however, the upper bounds in known inequalities can only result in the inequality \eqref{wsga}. Our approach here relies on the decoupling technique thanks to the special structure in the martingale transform.

Let us consider the discrete time martingale $x_n\in L_\8(\Om,\fx_n)\rtimes G$ given by
\begin{equation}\label{gmar}
   x_n=\sum_{k=0}^{n-1}\sum_{g\in G,j\in \nz} [f_g^j(\bt_k(g)) dB_{k+1}^j]\rtimes \la(g)
\end{equation}
where $f_g^j$ is a continuous function, for any $j\in \nz$, $(B_{k}^j)_k$ is a martingale with independent martingale differences $dB_{k+1}^j=B_{k+1}^j-B_{k}^j$. In what follows we will simply write $\sum_{g,j}$ instead of $\sum_{g\in G,j\in \nz}$ and this always means a finite sum.
\begin{lemma}[Decoupling]\label{deco}
  Suppose $\bt_k(g)$ is measurable with respect to $\si(B_m^j, m\le k, j\in \nz)$ for $g\in G$ and $k=0,\cdots, n-1$ and $(\td{B}_k^j)$ an independent copy of $(B_k^j)$. Then for $2\le p<\8$,
  \[
  \Big\|\sum_{k=0}^{n-1}\sum_{g,\ell}[f_g^\ell(\bt_k(g)) dB_{k+1}^\ell]\rtimes  \la(g)\Big\|_p^p\le 4^p \Big\|\sum_{k=0}^{n-1}\sum_{g,\ell} [f_g^\ell(\bt_k(g)) d\td{B}_{k+1}^\ell]\rtimes \la(g)\Big\|_p^p.
  \]
\end{lemma}
\begin{proof}
To shorten the notation, we simply write $\bt_k$ for $\bt_k(g)$. Consider independent random selectors $\de_k,k=0,\cdots,n$ with $\ez(\de_k)=1/2$. Define $\Delta=\{j\in \{0,\cdots, n\}: \de_j=1\}$. Then $\ez(\de_j(1-\de_k))=1/4$ for $j\neq k$. The left-hand side is
\begin{align*}
  A&:=\Big\|\sum_{g,\ell} \sum_{j,k=0}^{n}1_{\{j=k+1\}} [f_g^\ell(\bt_k) dB_{j}^\ell]\rtimes \la(g)\Big\|_p^p \\
  &= 4^p \Big\| \sum_{g,\ell}\sum_{j,k=0}^{n}1_{\{j=k+1\}}\ez_\de \de_k(1-\de_j) [f_g^\ell(\bt_k) dB_{j}^\ell]\rtimes \la(g)\Big\|_p^p.
\end{align*}
By Jensen's inequality, we have
\begin{align*}
  A &\le 4^p \ez_\de  \Big\| \sum_{g,\ell}\sum_{j,k=0}^{n} 1_{\{j=k+1\}} \de_k(1-\de_j)[f_g^\ell(\bt_k) dB_{j}^\ell]\rtimes \la(g) \Big\|_p^p\\
&\le 4^p \ez_\de \Big\| \sum_{g,\ell} \sum_{k\in \Delta,j\notin \Delta} 1_{\{j=k+1\}} [f_g^\ell(\bt_k) dB_{j}^\ell]\rtimes \la(g) \Big\|_p^p.
\end{align*}
Since $\bt_k$ and $dB_{k+1}^\ell$ are independent for all $\ell$, and taking expectation of $\bt_k$'s commutes with the group action, we have
\[
A \le 4^p \ez_\de \Big\| \sum_{g,\ell} \sum_{k\in \Delta,j\notin \Delta} 1_{\{j=k+1\}} [f_g^\ell(\bt_k) d\td{B}_{j}^\ell]\rtimes \la(g) \Big\|_p^p
\]
for any independent copy $(\td{B}_k^j)$ of $(B_k^j)$. We may and do fix a realization of $\de$ and thus fix a partition $\Delta_0\sqcup \Delta_0^c=\{0,\cdots, n\}$ so that
\begin{equation}\label{eqdc}
  A \le 4^p \Big\| \sum_{g,\ell}\sum_{k\in \Delta_0,k+1\notin \Delta_0} [f_g^\ell (\bt_k) d\td{B}_{k+1}^\ell]\rtimes \la(g) \Big\|_p^p.
\end{equation}
Now the von Neumann algebra has been enlarged to $(L_\8(\Om,\fx)\overline\otimes L_\8(\td{\Om},\td{\fx})) \rtimes G$ by diagonal extension of the group action. Here $(\td{\Om},\td{\fx})$ is generated by $(\td{B}_{k}^\ell)$. Let
\[
id\otimes E_{\td{\Delta_0^c}}: L_\8(\Om,\fx)\overline{\otimes } L_\8(\td{\Om},\td{\fx}) \to L_\8(\Om,\fx)\overline{\otimes} L_\8(\si(\td{B}_{k+1}^\ell:\ell\in\nz, k+1\in \Delta_0^c))
 \]
denote the conditional expectation. Note that $d\td{B}_{k+1}$'s are mean zero. Then we may rewrite \eqref{eqdc} as
\[
A\le 4^p \Big\| \Big(id\otimes E_{\td{\Delta_0^c}} \big(\sum_{k\in \Delta_0}\sum_{k+1=1}^n\sum_{g,\ell} [f_g^\ell(\bt_k) d\td{B}_{k+1}^\ell]\big)\Big)\rtimes \la(g) \Big\|_p^p.
\]
Observing that $k\in \Delta_0$ if and only if $k+1\in \Delta_0+1$, we have
\[
A\le 4^p \Big\| id\otimes E_{\td{\Delta_0+1}}\Big(id\otimes E_{\td{\Delta_0^c}} \big(\sum_{k=0}^{n-1}\sum_{g,\ell} [f_g^\ell(\bt_k) d\td{B}_{k+1}^\ell]\big)\Big)\rtimes \la(g) \Big\|_p^p,
\]
where $id\otimes E_{\td{\Delta_0+1}}$ is defined similarly to $id\otimes E_{\td{\Delta_0^c}}$ as above. Since conditional expectations extend to contractions on $L_p$, the proof is complete.
\end{proof}
\begin{rem}\label{decx}
  The general decoupling argument is a very powerful tool in various applications. In fact, a more general version of Lemma \ref{deco} holds. Namely, we can remove the condition that $dB_{k+1}^j$'s are martingale differences, only require them to be independent from $\bt_k(g)$. The proof follows the general decoupling technique developed for $U$-statistics due to de la Pe\~na \cite{dlP92}; see also the proof of \cite[Theorem 3.1.1]{dG}. We refer the interested reader to the monograph \cite{dG} for an extensive discussion of decoupling methods. We keep the current version for simplicity.
\end{rem}
Let us denote by $\td{x}_n$ the decoupled version of $x_n$, i.e.,
\[
\td{x}_n=\sum_{k=0}^{n-1}\sum_{g,j} [f_g^j(\bt_k(g)) d\td{B}_{k+1}^j]\rtimes \la(g).
\]
Consider the von Neumann subalgebra $\nx= L_\8(\Om,\fx)\rtimes G$. By the noncommutative Rosenthal inequality proved in \cite{JZ11}, we have for $2\le p<\8$,
\[
\|\td{x}_n\|_p\le C\max\Big\{ \sqrt{p}\Big\|\Big(\sum_{j=1}^n E_\nx(d\td{x}_j^*d\td{x}_j + d\td{x}_jd\td{x}_j^*) \Big)^{1/2}\Big\|_p, ~p\Big(\sum_{j=1}^n\|d\td{x}_j\|_p^p\Big)^{1/p}\Big\}
\]
where $d\td{x}_j=\sum_{g,\ell} [f_g^\ell(\bt_{j-1}(g)) d\td{B}_{j}^\ell]\rtimes \la(g)$ for $j=1,\cdots, n$. Now it is crucial to observe that
\[
E_\nx(d\td{x}_j^*d\td{x}_j + d\td{x}_jd\td{x}_j^*) = E_{j-1}(d{x}_j^*d{x}_j + d{x}_jd{x}_j^*),
\]
and $\|d\td{x}_j\|_p=\|d{x}_j\|_p$, where $dx_j=x_j-x_{j-1}$ is the martingale difference. We have then for $2\le p<\8$,
  \[
  \|x_n\|_p\le C\max\Big\{ \sqrt{p}\|(\sum_{k=1}^{n} E_{k-1}( dx_{k}^*dx_{k}+ dx_{k}dx_{k}^*))^{1/2} \|_p, ~ p(\sum_{k=1}^{n-1}\|dx_k\|_p^p)^{1/p}\Big\}.
  \]
In other words,
\begin{equation}\label{ros}
  \|x_n\|_p\le C\max\{\sqrt{p}\|x_n\|_{h_p^c}, \sqrt{p}\|x_n\|_{h_p^r}, p\|x_n\|_{h_p^d}\}.
\end{equation}
Now apply \eqref{ros} to the discretized martingale \eqref{dmar}. By the facts on Hardy spaces presentated in Section \ref{hard},
\[
\|v(x)\|_{h_p^c([0,L])}=\|\lge v(x),v(x)\rge_L^{1/2}\|_p = \lim_{n\to\8} \|M_n(x)\|_{h_p^c}.
\]
Since $v(x)$ is driven by Brownian motion, it has continuous path and is of vanishing variation, i.e.,
\[
\|v(x)\|_{h_p^d([0,L])}=\lim_{n\to \8} \|M_n(x)\|_{h_p^d}=0.
\]
See \cite{JP,JZ12} for more details. Combining things together, we have shown the following result.
\begin{lemma}[Burkholder--Davis--Gundy inequality]\label{bdg}
Let $v_t(x)=\pi_{t}T_{L-t}(x)$ be the martingale associated to $x\in LG$ as before, then for $2\le p<\8$,
\[
\|v_L(x)-v_0(x)\|_p\le C\sqrt{p} \max\{\|v(x)\|_{h_p^c([0,L])}, \|v(x)\|_{h_p^r([0,L])}\}.
\]
\end{lemma}

Let $\td{\pi}_t=\pi_{L-t}$, $\td\fx_{[t}=\fx_{L-t}$ and $E_{[t}=E_{L-t}$ for $t<L$. Define $n_t(x)=\pi_{L-t}(T_t x)$. It is easy to check that $(\td\pi_t,\td\fx_{[t})$ is a reversed Markov dilation, i.e., for $s<t$,
\[
E_{[t}\td\pi_s x = E_{L-t} \pi_{L-s}(x)=\pi_{L-t}T_{t-s}x=\td\pi_t T_{t-s}x.
 \]
It follows that $(n_t(x), E_{[t})$ is a reversed martingale and $v_L(x)-v_0(x)=n_0(x)-n_L(x)$
\begin{proof}  [Proof of Theorem \ref{mthe}]
The proof follows the same idea as for \cite[Theorem 4.4]{JZ12}. We give a sketch for completeness. By approximation, we may assume $x=\sum_{g\in G} x_g \la(g)$ is a finite linear combination. By replacing $x$ by $x-E_{\Fix}x$, we may assume $E_{\Fix}x=0$. It follows that $\lim_{L\to\8}\|T_L x\|_p=0$. The $\Ga_2$-criterion implies uniform boundedness for $\Ga(T_rx, T_rx)$ for $r\ge0$ in $L_p(LG)$. By \cite[Lemma 4.3]{JZ12} and Lemma \ref{bdg}, we have
\begin{align*}
  &\|\pi_Lx\|_p-\|\pi_0 T_Lx\|_p \le \|\pi_L x-\pi_0 T_Lx\|=\|n_0(x)-n_L(x)\|_p\\
  &\le C\sqrt{p}\max\Big\{ \Big\|\int_0^L \td\pi_{r}(\Ga(T_rx,T_rx))\dd r\Big\|_{p/2}^{1/2},~\Big\|\int_0^L \td\pi_{r}(\Ga(T_rx^*,T_rx^*))\dd r\Big\|_{p/2}^{1/2}\Big\}.
\end{align*}
Then the $\Ga_2$-criterion gives $\Ga(T_tx,T_tx)\le e^{-2\al t}T_t\Ga(x,x)$; see \cite[Lemma 4.6]{JZ12}. Since $\td\pi_t$ and $T_t$ are contractions, we have
\[
\Big\|\int_0^L \td{\pi}_{r}(\Ga(T_{r}x, T_{r}x))\dd r \Big\|_{p/2}^{1/2} \le \Big(\frac1{2\al}-\frac1{2\al e^{2\al L}} \Big)^{1/2}\|\Ga(x,x)\|_{p/2}^{1/2}.
\]
Similar inequality holds for $x^*$. Since $\pi_t$ is trace preserving $*$-homomorphism, $\|\pi_t x\|_p=\|x\|_p$ for all $t\ge0$ and $x\in LG$. We also have $\lim_{L\to\8}T_Lx =E_{\Fix} x=0$ in $L_p(LG)$. Now sending $L\to \8$ completes the proof.
\end{proof}
\begin{proof}[Proof of Corollary \ref{sgau}]
  Notice that the group algebra $\cz G$ is weakly dense in $LG$. The assertion follows verbatim the proof of \cite[Corollary 4.8]{JZ12}.
\end{proof}
\begin{proof}[Proof of Corollary \ref{tran}]
  By \cite[Theorem 1.3]{Z13}, the $L_p$ Poincar\'e inequalities with constants $\sqrt{p}$ implies the second inequality in the assertion. But Corollary \ref{sgau} says that Bakry--Emery's condition gives the subgaussian growth for $2\le p<\8$, and therefore yields the second inequality. The first one is even simpler. See Proposition 3.14 and Corollary 3.19 in \cite{JZ12}.
\end{proof}
The strategy we used here can be applied to other settings as long as the martingale obtained from the Markov dilation can be approximated by a martingale transform like \eqref{dmar}. Let us consider an application to the Lindblad operator in quantum dynamical system; see \cite{Lin,Par}.
Let $(a_j)_{j=1}^m\subset M_n$ be a family of mutually commuting Hermitian matrices, where $M_n$ is the matrix algebra of dimension $n^2$. Define $A$ acting on $M_n$ by
\[
A(x)= \sum_{j=1}^m x a_j^2+ a_j^2 x -2 a_jxa_j.
\]
Consider the semigroup $T_t=e^{-tA}$ acting on $M_n$ generated by $A$. Then we have the following result.
\begin{theorem}
  Suppose there exists $\al>0$ such that $\Ga_2^A(x,x)\ge \al \Ga^A(x,x)$ for $x\in M_n$. Then for $2\le p<\8$,
  \begin{equation}
      \|x-E_{\Fix} x\|_p\le C \sqrt{p/\al} \max\{\|\Ga^A(x,x)^{1/2}\|_p, \|\Ga^A(x^*,x^*)^{1/2}\|_p\}.
  \end{equation}
\end{theorem}
\begin{proof}[Sketch of proof]
  The proof follows that of Theorem \ref{mthe} with appropriate modification. We sketch the main steps here. Let $u_t= \exp(i\sum_{j=1}^mB_t^j a_j)$, where $B_t^j$'s are independent Brownian motions with generator $d^2/dx^2$. Since we assumed $a_j$'s are mutually commuting, by Ito's formula, $u_t$ satisfies the following stochastic differential equation
  \[
  \dd u_t= -\sum_{k=1}^m a_k^2 u_t \dd t+ i\sum_{k=1}^m u_ta_k \dd B_t^k.
  \]
  Let $\pi_t x=u_t^* x u_t$ for $x\in M_n$. Then it was shown in \cite{JRS}
  that $\pi_t$ is a Markov dilation for $T_t$, i.e., $E_s\pi_t x= \pi_s T_{t-s} x$ for $s<t$.

  Fix $L>0$. Let $v_t(x)=\pi_t T_{L- t}x$. It is a martingale for $0<t<L$. By Ito's formula,
  \[
  \pi_L x - \pi_0T_L x = i\sum_k \int_0^L u_s^* (-a_k T_{L-t}x +T_{L-t}x a_k)u_s \dd B_s^k .
  \]
  Then we can discretize the stochastic integral and apply a decoupling argument to find the  BDG inequality for $v_t(x)$, as what we did in the proof of Theorem \ref{mthe}. Note that $n_t(x):=\pi_{L-t}(T_t x)$ is a reversed martingale and $n_0(x)-n_L(x)=\pi_Lx-T_Lx$. Combining \cite[Lemma 4.3]{JZ12} with the $\Ga_2$-criterion, we arrive at the assertion.
\end{proof}

\section{Spectral gap and $L_p$ Poincar\'e inequalities}
In the classical diffusion setting, if the probability measure is non-atomic, it is known that the $L_2$ Poincar\'e inequality (a.k.a. the spectral gap inequality) implies the $L_p$ Poincar\'e inequalities with constant $Cp$ for $2\le p<\8$; see e.g. \cite[Proposition 2.5]{Mi09}. In this section, we show that such implication still holds in the noncommutative setting under some conditions. A crucial ingredient in the commutative theory is the Leibniz rule \eqref{leib} for the gradient form, which is not available in general non-diffusion setting. The remedy relies on two ingredients. The first is the derivation property of $\Ga$ proved in \cite{CS03}*{Theorem 8.2} (see also \cite{JRS}), i.e., there exists a closable derivation $\delta$ on $L_2(\nx, \tau)$, such that
\begin{equation}\label{deri}
  \tau(\Ga(x,x))=\tau(\de(x)^*\de(x))
\end{equation}
for $x\in\Dom (A^{1/2})\cap\nx$ (the domain of the Dirichlet form). The second is the following assumption
\begin{equation}\label{compa}
  C^{-1}\|\de(x)\|_p \le \max\{\|\Ga(x,x)^{1/2}\|_p, \|\Ga(x^*,x^*)^{1/2}\|_p\}\le C\|\de(x)\|_p
\end{equation}
for all $1<p<\8$.

\begin{rem}
The assumption \eqref{compa} is verified in Junge--Ricard--Shlyakhtenko \cite{JRS} for any standard semigroup which satisfies \eqref{ncdif}. Indeed, if interpreted correctly such a derivation can be constructed in an ultra-power using approximating $\delta_{\al}$ as in Section 3 of \cite{CS03} for all completely positive self-adjoint trace preserving semigroups.  The proof uses free probability theory to construct Markov dilations. We will not pursue this point here further.
\end{rem}

The following technical result is standard.
\begin{lemma}\label{gfin}
 Let $1\le p<\8$. Then
\[
 \|\Ga(x,y)\|_p\le \|\Ga(x,x)\|_p^{1/2}\|\Ga(y,y)\|_p^{1/2}.
\]
\end{lemma}
\begin{proof}
The case $p=1$ was proved in \cite[Corollary 4.8]{JZ12}. The general case follows from the same argument with the help of H\"older's inequality.
\end{proof}
\begin{lemma}\label{gamin}
  Let $f_1,\cdots, f_m, g_1,\cdots, g_n\in \Dom(A)$ be self-adjoint elements. Assume \eqref{compa} holds. Then
  \[
  \tau[\Ga(f_1\cdots f_m, g_1\cdots g_n)] \le C\sum_{i=1}^m\sum_{j=1}^n \|\Ga(f_i, f_i)^{1/2}\|_{m+n}\|\Ga(g_j, g_j)^{1/2}\|_{m+n} \prod_{k\neq i, \ell\neq j} \|f_k\|_{m+n}\|g_\ell\|_{m+n}.
  \]
\end{lemma}
\begin{proof}
  By \eqref{deri} and the derivation property, we have
  \[
  \tau[\Ga(f_1\cdots f_m, g_1\cdots g_n)] = \sum_{i,j}\tau[(f_1\cdots f_{i-1}\de(f_i)f_{i+1}\cdots f_m)^*(g_1\cdots g_{j-1}\de(g_j)g_{j+1}\cdots g_n)].
  \]
  Using H\"older's inequality, we obtain
  \begin{align*}
  &\tau[\Ga(f_1\cdots f_m, g_1\cdots g_n)] \\
  \le ~&\sum_{j,k} \|f_1\|_{m+n}\cdots \|f_{j+1}\|_{m+n}\|\de(f_j)\|_{m+n}\|f_{j-1}\|_{m+n}\cdots \|f_m\|_{m+n}\cdot\\
  &\|g_1\|_{m+n}\cdots \|g_{k-1}\|_{m+n}\|\de(g_k)\|_{m+n}\|g_{k+1}\|_{m+n}\cdots\|g_n\|_{m+n}.
  \end{align*}
  Now using \eqref{compa}, we complete the proof.
\end{proof}

Recall that $T_t$ is ergodic if the fixed point algebra of $T_t$ is trivial. Thus $E_{\Fix} (x) = \tau(x)$.
\begin{theorem}\label{spg1}
 Let $T_t$ be an ergodic standard semigroup acting on a diffuse probability space $(\nx,\tau)$ which satisfies \eqref{compa}. Suppose the generator $A$ of $T_t$ has a spectral gap: For $f\in \nx$,
$$\|f-\tau(f)\|_2\le C\max\{\|\Ga(f,f)^{1/2}\|_2, \|\Ga(f^*,f^*)^{1/2}\|_2\}.$$
Then for all even integer $p\ge 2$ and all $f\in \nx$,
\[
 \|f-\tau(f)\|_p\le C'p \max\{\|\Ga(f,f)^{1/2}\|_p, \|\Ga(f^*,f^*)^{1/2}\|_p\}.
\]
\end{theorem}
\begin{proof}
 Let $g\in\nx$ be a self-adjoint element and $p=2q$ be an even integer. Without loss of generality we may assume $\tau(g)=0$. Since $\nx$ is diffuse, there is no minimal projection in $\nx$. Let $g=\int_\rz t \dd E_t$ be the spectral decomposition, where $\dd E$ is the spectral measure of $g$. Then the scalar spectral measure $\tau\circ \dd E$ of $g$ is non-atomic. We can find a function $\sgn(\cdot): \spec(g)\to\{\pm 1\}$ such that $\tau[\sgn(g)g^{p/2}]=0$, where $\spec(g)$ denotes the spectrum of $g$. Let $f=\sgn(g)g^{p/2}$. Applying the spectral gap inequality on $f$, we have
\begin{equation}\label{spga}
  \tau(g^{2q})\le C^2 \tau[\Ga(g^{q}, g^q)].
\end{equation}
By Lemma \ref{gamin},
\[
\tau[\Ga(g^{q}, g^q)]\le C'q^2 \|g\|_{2q}^{2q-2}\|\Ga(g,g)^{1/2}\|_{2q}^2
\]
Hence,
\begin{equation}\label{sfad}
 \|g\|_{2q}\le C'q \|\Ga(g,g)^{1/2}\|_{2q}.
\end{equation}
For general mean zero element $f\in \nx$, write $f=\Re(f)+i \Im(f)$, where $\Re(f)=\frac{f+f^*}{2}$ and $\Im(f)=\frac{f-f^*}{2i}$. Using the triangle inequality and \eqref{sfad}, we obtain
\[
 \|f\|_p\le \|\Re(f)\|_p+\|\Im(f)\|_p\le C'q (\|\Ga(\Re(f),\Re(f))^{1/2}\|_p+\|\Ga(\Im(f),\Im(f))^{1/2}\|_p).
\]
By Lemma \ref{gfin}, we find
\begin{align*}
  \|\Ga(\Re(f),\Re(f))^{1/2}\|_p &=\|\Ga(\Re(f),\Re(f))\|_q^{1/2}\\
  &=\frac12[\|\Ga(f,f)\|_q+2\|\Ga(f,f^*)\|_q+\|\Ga(f^*,f^*)\|_q ]^{1/2}\\
&\le \frac12[\|\Ga(f,f)\|_q+2\|\Ga(f,f)\|_q^{1/2} \|\Ga(f^*,f^*)\|_q^{1/2}+\|\Ga(f^*,f^*)\|_q ]^{1/2}\\
&
\le\frac1{{2}} [\|\Ga(f,f)\|_q^{1/2}+\|\Ga(f^*,f^*)\|_q^{1/2}].
\end{align*}
Similar argument applies to $\|\Ga(\Im(f),\Im(f))^{1/2}\|_p$ and the proof is complete.
\end{proof}
Note that the diffuse and ergodic assumptions are indispensable in the above argument. We provide some results without these assumptions in the following.
\begin{theorem}\label{2pow}
   Let $T_t$ be a standard semigroup acting on a probability space $(\nx,\tau)$ with \eqref{compa} such that
   \begin{equation}\label{extau}
   \|E_{\Fix} g\|_2\le C_1\tau(g)
    \end{equation}
    for all $g\ge0$. Suppose the spectral gap inequality holds: For $f\in \nx$,
$$\|f-E_{\Fix}(f)\|_2\le C_2\max\{\|\Ga(f,f)^{1/2}\|_2, \|\Ga(f^*,f^*)^{1/2}\|_2\}.$$
Then we have for all $f\in \nx$ and $k\in \nz$,
\[
 \|f-E_{\Fix}(f)\|_{2^{k}}\le C_3 2^{k} \max\{\|\Ga(f,f)^{1/2}\|_{2^k}, \|\Ga(f^*,f^*)^{1/2}\|_{2^k}\}.
\]
\end{theorem}
\begin{proof}
  By the same argument as for Theorem \ref{spg1}, it suffices to consider the self-adjoint element $f$. Since $\Ga(f-E_{\Fix}f, f-E_{\Fix}f)=\Ga(f,f)$, we may assume $E_{\Fix} f=0$. Note that $k=1$ is the spectral gap inequality. We proceed by induction. Assume
  \[
  \|f\|_{2^{k}}\le A_k 2^{k} \|\Ga(f,f)^{1/2}\|_{2^k},
  \]
  where $A_k$ is the best constant. Applying the spectral gap inequality to $f^{2^k}$ and using the assumption \eqref{extau}, we have
  \[
  \|f\|_{2^{k+1}}^{2^{k+1}}\le C_2^2\tau[\Ga(f^{2^k},f^{2^k})]+C_1^2\tau(f^{2^k})^2.
  \]
  By Lemma \ref{gamin} and the induction hypothesis,
  \[
  \|f\|_{2^{k+1}}^{2^{k+1}}\le CC_2^2 2^{2k}\|\Ga(f,f)^{1/2}\|_{2^{k+1}}^2\|f\|_{2^{k+1}}^{2^{k+1}-2}+C_1^2 A_k^{2^{k+1}} 2^{k 2^{k+1}}\|\Ga(f,f)^{1/2}\|_{2^{k}}^{2^{k+1}}=: I+II.
  \]
  Suppose $I\le II$. Since $\tau(1)=1$, we have
  \[
  \|f\|_{2^{k+1}}^{2^{k+1}}\le 2C_1^2 A_k^{2^{k+1}} 2^{k 2^{k+1}}\|\Ga(f,f)^{1/2}\|_{2^{k+1}}^{2^{k+1}}.
  \]
  It follows that $A_{k+1}\le (\sqrt{2}C_1)^{1/2^{k}}2^{-1} A_k$. Suppose $II\le I$. Then
  \[
  \|f\|_{2^{k+1}}^{2^{k+1}}\le 2CC_2^2 2^{2k}\|\Ga(f,f)^{1/2}\|_{2^{k+1}}^2\|f\|_{2^{k+1}}^{2^{k+1}-2}.
  \]
  We have
  \[
  \|f\|_{2^{k+1}}\le \frac{1}{\sqrt{2}}\sqrt{C}C_2 2^{k+1}\|\Ga(f,f)^{1/2}\|_{2^{k+1}}.
  \]
  Hence $A_{k+1}\le \sqrt{C}C_2/\sqrt{2}$. It follows that
  $$A_{k+1}\le \max\{(\sqrt{2}C_1)^{1/2^{k}}2^{-1} A_k, \sqrt{C} C_2/\sqrt{2}\}\le \max\{(\sqrt{2}C_1)^{1/2^{k}} A_k, \sqrt{C}C_2/\sqrt{2}\}.
   $$
  Note that we may assume without loss of generality $\sqrt{2} C_1\ge 1$ and  $C_1\ge \sqrt{C}/2$. Since we may take $A_1=C_2$, inductively we have $(\sqrt{2}C_1)^{1/2^{k}} A_k\ge \sqrt{C}C_2/\sqrt{2}$ for all $k\in\nz$. Let
  $$B_k= C_2\prod_{j=1}^{k-1} (\sqrt{2}C_1)^{1/2^{j}}.
  $$
  Since $\log B_k \le \log(\sqrt{2}C_1C_2)$, $A_k\le B_k$ is uniformly bounded. The proof is complete.
\end{proof}
To state a result for arbitrary $p$, let us recall the $L_p$ regularity of Dirichlet forms due to Olkiewicz and Zegarlinski \cite[Theorem 5.5]{OZ99}. In our context, their result implies (\cite[(22)]{OZ99}) that
\[
 \tau(f^{p/2}Af^{p/2}) \le \frac{p^2}{4(p-1)} \tau(f^{p-1} A f)
\]
for positive $f$ in the domain of the Dirichlet form on the right-hand side and $1<p<\8$. By \cite{CS03}, we know that $\tau(fAg)=\tau[\Ga(f, g)]$. It follows that
\begin{equation}\label{gampp}
 \tau[\Ga(f^{p/2},f^{p/2})]\le \frac{p^2}{4(p-1)}\tau[\Ga(f,f^{p-1})]
\end{equation}
for $f\ge0$.
\begin{theorem}
  Under the assumptions of Theorem \ref{2pow}, for $p\ge 2$, there exists a finite set $F_p\subset[1,2)$ determined by $p$, such that for every self-adjoint element $f\in\nx$ with $E_{\Fix}(f)=0$, we have
  \begin{equation}\label{pcral}
     \|f \|_{p}\le C' \max\left\{ \max_{\al\in F_p} ~ p^{1/\al}\|\Ga(|f|^{\al},|f|^{\al})^{1/2}\|_{p/\al}^{1/\al}, ~p\|\Ga(f,f)^{1/2}\|_p\right\}.
  \end{equation}
\end{theorem}
\begin{proof}
 We argue by induction on $n$ for $2^n\le p\le 2^{n+1}$. Let $2\le p<\8$. By the spectral gap inequality, we have
\[
 \|f\|_p^p=\||f|^{p/2}\|_2^2\le C_2^2\tau[\Ga(|f|^{p/2}, |f|^{p/2})] + \|E_{\Fix}(|f|^{p/2})\|_2^2.
\]
Using \eqref{deri} and \eqref{gampp}, we have
\begin{align}\label{spgap}
 \tau[\Ga(|f|^{p/2}, |f|^{p/2})]\le \frac{p^2}{4(p-1)}\tau[\de(|f|)^*\de(|f|^k|f|^\al)],
\end{align}
where $[p-2]=k$, $p=k+1+\al$ and $1\le \al<2$. Using the derivation property, H\"older's inequality and \eqref{compa},
\begin{align*}
 &\tau[\de(|f|)^*\de(|f|^k|f|^\al)]=\sum_{j=1}^k\tau[\de(|f|)^*|f|^{k-j}\de(|f|)|f|^{j-1}|f|^{\al}]+\tau[\de(|f|)^*|f|^k\de(|f|^\al)]\\
&\le \sum_{j=1}^k \|\de(|f|)^*\|_p\|f\|_p^{k-j}\|\de(|f|)\|_p\|f\|_p^{j-1+\al} + \|\de(|f|)^*\|_p \|f\|_p^k\|\de(|f|^\al)\|_{p/\al}\\
&\le C^2[ k\|\Ga(|f|,|f|)^{1/2}\|_p^2 \|f\|_p^{k-1+\al}+ \|\Ga(|f|,|f|)^{1/2}\|_p\|f\|_p^k \|\Ga(|f|^\al, |f|^{\al})^{1/2}\|_{p/\al}].
\end{align*}
Let $\lx_{p,\al}=\|\Ga(|f|^\al,|f|^\al)^{1/2}\|_{p/\al}\|f\|_p^{p/2-\al}$. Noticing the relationship among $p,k,\al$, we find
\begin{equation}\label{lpal}
   \tau[\de(|f|)^*\de(|f|^k|f|^\al)]\le C^2[(p-1-\al) \lx_{p,1}^2+\lx_{p,1}\lx_{p,\al}] \le C^2(p-\al)\max\{\lx_{p,\al}^2, \lx_{p,1}^2\}.
\end{equation}
On the other hand, by assumption \eqref{extau}, we have for $2\le p\le 4$,
\begin{align*}
 \|E_{\Fix} (|f|^{p/2})\|_2^2 &\le C_1^2\tau(|f|^{p/2})^2\le C_1^2\|f\|_2^p\\
&\le C_1^2 C_2^p \|\Ga(f,f)^{1/2}\|_2^p\le C_1^2 C_2^p \|\Ga(f,f)^{1/2}\|_p^p.
\end{align*}
Plugging into \eqref{spgap}, we have
\[
 \|f\|_p^p\le \frac{C^2C_2^2 p^2}{4} \max\{\lx_{p,\al}^2, \lx_{p,1}^2\}+C_1^2 C_2^p \|\Ga(f,f)^{1/2}\|_p^p=:I+II.
\]
If $I\le II$, we have $\|f\|_p\le 2^{1/p}C_1^{2/p}C_2 \|\Ga(f,f)^{1/2}\|_p$. Suppose $II\le I$. Let $F_p=\{1,\al\}$. We may find $\al_0\in F_p$ such that $\lx_{p,\al_0}^2=\max\{\lx_{p,\al}^2, \lx_{p,1}^2\}$. It follows that
\begin{equation*}
   \|f\|_p^{2\al_0}\le \frac{C^2C_2^2 p^2}{2} \|\Ga(|f|^{\al_0},|f|^{\al_0})^{1/2}\|_{p/\al_0}^2.
\end{equation*}
Hence, we find
\begin{equation}\label{alph0}
 \|f\|_p\le \frac{(CC_2)^{1/\al_0}p^{1/\al_0}}{2^{1/(2\al_0)}}\|\Ga(|f|^{\al_0},|f|^{\al_0})^{1/2}\|_{p/\al_0}^{1/\al_0} \le C' \max_{\al
  \in F_p} p^{1/\al}\|\Ga(|f|^{\al},|f|^{\al})^{1/2}\|_{p/\al}^{1/\al}.
\end{equation}
We have proven \eqref{pcral} for $2\le p\le 4$. Assume \eqref{pcral} holds for $2^{n-1}\le p\le 2^n$ and let $A_n$ denote the best constant. By assumption \eqref{extau} and the induction hypothesis,
\begin{align*}
  &\|E_{\Fix} (|f|^{p/2})\|_2^2 \le C_1^2\|f\|_{p/2}^p \\
&\le C_1^2 A_n^p\max\Big\{ \max_{\al\in F_{p/2}} \Big(\frac{p}2\Big)^{p/\al}\|\Ga(|f|^{\al}, |f|^{\al})^{1/2}\|^{p/\al}_{p/(2\al)}, ~\Big(\frac{p}2\Big)^p\|\Ga(f,f)^{1/2}\|_{p/2}^p\Big\}=:III.
\end{align*}
Combining with \eqref{spgap} and \eqref{lpal}, we get
\[
\|f\|_p^p\le \frac{C^2C_2^2 p^2}{4} \max\{\lx_{p,\al}^2, \lx_{p,1}^2\} + III = I+III.
\]
If $I\ge III$, we get \eqref{alph0} as above. In this case, we may take $F_p=\{1,\al\}$ and
\[
A_{n+1}\le \sup_{1\le \al\le 2}{2^{-1/(2\al)}}{(CC_2)^{1/\al}}.
\]
Now suppose $I\le III$. Then
\begin{align*}
  \|f\|_p &\le 2^{1/p}C_1^{2/p} A_n\max\Big\{ \max_{\al
  \in F_{p/2}} \Big(\frac{p}2\Big)^{1/\al}\|\Ga(|f|^{\al}, |f|^{\al})^{1/2}\|^{1/\al}_{p/(2\al)}, ~\frac{p}2 \|\Ga(f,f)^{1/2}\|_{p/2}\Big\}\\
  &\le 2^{1/p}C_1^{2/p} A_n\max\Big\{ \max_{\al
  \in F_{p/2}} {p}^{1/\al}\|\Ga(|f|^{\al}, |f|^{\al})^{1/2}\|^{1/\al}_{p/\al}, ~{p} \|\Ga(f,f)^{1/2}\|_{p}\Big\}.
\end{align*}
In this case, we may take $F_p=F_{p/2}$ and $A_{n+1}\le 2^{1/p}C_1^{2/p} A_n$. Combining together, we may set $F_p=F_{p/2}\cup \{\al\}$ and $A_{n+1}\le \max\{\sup_{1\le \al\le 2}(CC_2/\sqrt{2})^\al, (2C_1^2)^{1/2^n}A_n\}$. We may assume without loss of generality $\sup_{1\le \al\le 2}(CC_2/\sqrt{2})^\al\le  (2C_1^2)^{1/2}A_1$ and $2C_1^2\ge1$. Thus inductively
\[
\sup_{1\le \al\le 2}(C_2/\sqrt{2})^\al\le (2C_1^2)^{1/2^n}A_n.
\]
By the same argument as for Theorem \ref{2pow}, $A_n$ is uniformly bounded and the proof is complete.
\end{proof}
\begin{rem}
  It is not difficult to check that the assumption \eqref{extau} is satisfied if the fixed point algebra of $T_t$ is finite dimensional. The constant $C_1$ depends on the dimension of the fixed point algebra and the trace on this algebra. In fact, finite dimensional von Neumann algebras are of the form $\oplus_{i=1}^r M_{n_i}$, where $M_{n_i}$ is the matrix algebra of dimension $n_i^2$. For simplicity, let us illustrate the case $\Fix=M_n$. For $x\in M_n$, $$\|x\|_2=\Big[\frac1{n}tr(x^*x)\Big]^{1/2}=\frac1{\sqrt{n}}\sum_{i=1}^n s_i^2,
  $$
  where $tr$ is the usual trace on $M_n$, and $s_i$'s are the singular values of $x$. Similarly, $\|x\|_1=\frac1n\sum_{i=1}^n s_i$. Then $\|x\|_2\le \sqrt{n}\|x\|_1$. Hence, for $g\in\nx$,
  \[
  \|E_{\Fix} g\|_2\le \sqrt{n}\|E_{\Fix}g\|_1\le \sqrt{n}\|g\|_1.
  \]
  The general form $\Fix=\oplus_{i=1}^r M_{n_i}$ is slightly more complicated and we leave it to the interested reader.
\end{rem}
\begin{rem}
  Although it looks complicated, the inequality \eqref{pcral} is actually consistent with that in the classical diffusion theory. To simplify our calculation, let us consider the one-variable functions and assume $\Ga(f,f)^{1/2}=|f'|$. Assume further that $|f|$ is differentiable and $\int f d\mu=0$. Then $f'(x)=0$ for $f(x)=0$. For example, $f(x)=x^2 \sgn(x)$ defined on the Gaussian probability space $(\rz,\ga)$ satisfies these conditions. Since $\Ga(f,f)^{1/2}=\Ga(|f|,|f|)^{1/2}$ in this setting, we only need to consider the first term in \eqref{pcral}. By H\"older's inequality, we get
  \[
  \|(|f|^\al) '\|_{p/\al}^{1/\al}\le \al^{1/\al}\||f|^{\al-1}|f|'\|_{p/\al}^{1/\al} \le \al^{1/\al}\|f\|_p^{1-1/\al}\||f|'\|_p^{1/\al}.
  \]
  After choosing the optimal $\al$, we have
  \[
  \|f\|_p\le C \al p \||f|'\|_{p}=C \al p \|f'\|_{p},
  \]
  which is exactly the classical result deduced from the spectral gap inequality as in \cite{Mi09}. In general, $|f|$ may not be differentiable at the zeros of $f$ even if $f$ is smooth. In this case, one may use a smoothening procedure by convolution to deduce similar results.
\end{rem}

\section{Examples and illustrations}
As explained above, the spectral gap may lead to the $L_p$ Poincar\'e inequalities with constant $Cp$ under certain conditions. Our first example illustrates that even in the classical diffusion setting one can not achieve $C\sqrt{p}$ assuming only the existence of spectral gap.

\begin{exam}[Spectral gap is not sufficient]\label{sgap}
Consider the double exponential distribution on $\rz$ given by $\mu(\dd x)=\frac12 e^{-|x|}\dd x$. There exists a semigroup $T_t$ which is symmetric on $L_2(\rz, \mu)$ with generator given by
$$-A= \frac{\dd^2}{\dd x^2}-\sgn(x)\frac{\dd}{\dd x}$$
on compactly supported smooth functions $f$ with $f'(0)=0$, where $\sgn(x)$ is the sign of $x$. Clearly such functions are dense in $L_2(\rz,\mu)$. It was shown in \cite{BL97} that $\mu$ satisfies the $L_2$ Poincar\'e inequality. However, it is easy to see that the $L_p$ Poincar\'e inequalities \eqref{pcr2} cannot hold by testing $f(x)=x$. By \eqref{impl}, the semigroup $(T_t)$ has to fail the Bakry--Emery $\Ga_2$-criterion. In this way, one can come up with a family of diffusion processes for which Bakry--Emery's condition fails. Indeed, let $\mu_\al(x)=\frac1{C_\al}e^{-|x|^\al}\dd x$ for $1\le \al<2$ on $\rz$ where $C_\al$ is a normalizing constant. Consider
$$-A_\al= \frac{\dd^2}{\dd x^2}-\al |x|^{\al-1}\sgn(x) \frac{\dd}{\dd x}$$
on compactly supported smooth functions $f$ with $f'(0)=0$. $-A_\al$ generates a symmetric semigroup $T_t^\al$ on $L_2(\rz,\mu_\al)$. The corresponding Markov process is a diffusion process. All these $T_t^\al$ for $1\le \al<2$ will fail \eqref{pcr2}, and thus fail Bakry--Emery's criterion. In fact, in this case, $\Ga(f,f)(x)=|f'(x)|^2$ but $\Ga_2(f,f)(x)=\al(\al-1)|x|^{\al-2} |f'(x)|^2$ for $x\neq 0$. Observe that we have only  $\Ga_2(f,f)\ge0$, but the spectral gap still exists by the same argument as for \cite[Lemma 2.1]{BL97}. Hence by, e.g., \cite[Proposition 2.5]{Mi09}, $T_t^\al$ satisfies the $L_p$ Poincar\'e inequalities with constants $Cp$.
\end{exam}
Our second example is meant to clarify the subgaussian behavior we discuss here via $L_p$ Poincar\'e inequalities is a condition on the semigroup (or its generator), not on the (noncommutative) probability space. This justifies the notion of subgaussian 1-cocycles.
\begin{exam}
  Consider the exponential distribution on $[0,+\8)$ given by $\mu(\dd x)=e^{-x}\dd x$. By \cite{KS85}, there is a conservative Markov semigroup which is symmetric in $L_2([0,\8),\mu)$ with generator $-A=x\frac{\dd^2}{\dd x^2}+(1-x)\frac{\dd}{\dd x}$. A calculation shows that
  $$\Ga_2(f,f)(x)- \Ga(f,f)(x)=(xf''(x)+f'(x)/2)^2+f'(x)^2/4\ge 0$$
  for all compactly supported smooth functions $f$. Since $-A$ generates a diffusion process, by \cite{AS94, JZ12}, we have \eqref{pcr2}. Note that the exponential distribution is not subgaussian in the sense of \cite{Ver} because $\|X\|_p=\Ga(p+1)^{1/p}\sim p$ where the law of $X$ is $\mu$. This means that the semigroups could satisfy the subgaussian Poincar\'e inequalities even though its invariant measure is not a subgaussian distribution. Roughly speaking, the gradient form in \eqref{pcr2} will provide another factor which compensates the factor $\sqrt{p}$. For instance, in our example here, if $f(x)=x$, then $\Ga(f,f)^{1/2}=\sqrt{x}$.
\end{exam}
\begin{rem}
  The above examples showed that the $L_p$ Poincar\'e inequalities provide more information than the moment estimates of probability measures. Indeed, the exponential distribution and the double exponential distribution have the same decay at $+\infty$. But there exist different semigroups such that the $L_p$ Poincar\'e inequalities \eqref{pcr2} may or may not hold.

  It is also interesting to compare \eqref{pcr2} and the log-Sobolev inequality in deducing concentration inequalities. On one hand, it is known (see \cite{AS94}) that log-Sobolev inequality implies \eqref{pcr2} in the classical diffusion setting while it was shown in \cite{JZ12} that in general non-diffusion situation, \eqref{pcr2} may still hold when the log-Sobolev inequality fails.  One can deduce concentration results from \eqref{pcr2}. On the other hand, although the spectral gap itself is not sufficient to give the $L_p$ Poincar\'e inequalities \eqref{pcr2} as shown in Example \ref{sgap}, Bobkov and Ledoux showed in \cite{BL97} that the exponential distribution satisfies a modified version of log-Sobolev inequality. From here, they proved concentration inequalities (see also \cite{BG}). It seems from the above discussion that the log-Sobolev inequality and the $L_p$ Poincar\'e inequalities are both useful in their own right and cannot entirely replace each other.
\end{rem}

Our theorems apply to a number of 1-cocycles on groups, including the free groups, finite cyclic groups, discrete Heisenberg groups, etc. See \cite[Section 5]{JZ12} for precise 1-cocycles on these groups and other examples. As explained in the Introduction, the Poincar\'e type inequalities there have been improved to the desirable form. We recapitulate and extend some interesting examples below for the reader's convenience. Recall that the Gromov form $K(g,h)=\frac12[\psi(g)+\psi(h)-\psi(g^{-1}h)]$.

\begin{exam}[Word length on free groups]\label{exfre}
Let $\fz_r$ be the free group with $r$ generators. Let $\psi(g)=|g|$ be the word length of $g\in \fz_r$ in the Cayley graph of $\fz_r$. By Haagerup's result \cite{Haa}, $\psi$ is conditionally negative. It was proved in \cite[Proposition 5.5]{JZ12} that $\Ga_2- \Ga$ is a positive semidefinite form on $\cz(\fz_r)$. Hence, the $\Ga_2$-criterion holds.
\end{exam}

\begin{exam}[Poisson semigroup on $L(\zz^r)$]
Let $\zz^r$ be the $r$ dimensional integer lattice. For $k\in \zz^r$, let $\psi(k)=\sum_{i=1}^r |k_i|$.  It is easy to check from the definition that $\psi$ is a cn-length function. The semigroup generated by $\psi$ is just the Poisson semigroup on the group von Neumann algebra $L(\zz^r)=L_\8(\tz^r)$. It was proved in \cite[Proposition 5.9]{JZ12} that $\Ga_2-\Ga$ is a positive semidefinite form from which the $\Ga_2$-criterion follows. The proof is based on the fact that $\zz$ is the free group with one generator.
\end{exam}

\begin{exam}[Discrete Heisenberg group]\label{exhei}
Let $H_3(\zz)$ be the discrete Heisenberg group over $\zz$. The multiplication here is given by
\begin{equation}\label{hemul}
(a,b,c)(a',b',c')=(a+a'+bc', b+b', c+c'), \quad (a,b,c), (a',b',c')\in H_3(\zz).
\end{equation}
Define $\psi(a,b,c)=|b|+|c|$, where $|b|$ is the absolute value of $b\in \zz$. We claim that $\psi$ is a cn-length function and satisfies $\Ga_2\ge \Ga$ on $\cz[H_3(\zz)]$. Indeed, it follows from the definition that $\psi$ is a cn-length function. By Example \ref{exfre}, the Gromov form $K$ associated to the word length on $\zz$ satisfies the condition that $(K(b,b')^2-K(b,b'))_{b,b'}$ is a positive semidefinite matrix. So does the Gromov form associated to $\psi$ here. It follows that $\psi$ verifies the $\Ga_2$-criterion. See the proof of \cite[Proposition 5.13]{JZ12} for details of the argument.
\end{exam}

The previous examples are all infinite groups. There are also interesting examples in finite groups.
\begin{exam}[Walsh systems]
Let $\zz_n$ denote the finite cyclic group and $\de_{x,y}$ denote the Kronecker delta function.
 Let $G= \zz_n^m=\zz_n\times\cdots\times \zz_n$ be the cartesian product of $m$ copies of $\zz_n$ with cn-length function
\[
\psi(x_1,\cdots,x_m)= m-\de_{x_1,0}-\cdots-\de_{x_m,0}, \quad (x_1,\cdots,x_m)\in \zz_n^m.
\]
It was shown in \cite[Proposition 5.12]{JZ12} that $\Ga_2-\frac{n+2}{2n}\Ga$ is a positive semidefinite form in $LG$. Thus \eqref{pcr3} follows. When $n=2$, this recovers Efraim and Lust-Piquard's Poincar\'e type inequalities for Walsh systems \cite{ELP}.
\end{exam}
\begin{exam}[Matrix algebras]
Let $H_3(\zz_n)$ be the discrete Heisenberg group over $\zz_n$. The multiplication is given by \eqref{hemul}. It was shown in \cite[Proposition 5.13]{JZ12} that $\Ga_2-\frac{n+2}{2n}\Ga$ is a positive semidefinite form in $L(H_3(\zz_n))$, the group von Neumann algebra of $H_3(\zz_n)$, where the semigroup $T_t$ acting on $L(H_3(\zz_n))$ is generated by the cn-length function $\psi(a,b,c)=2-\de_{b,0}-\de_{c,0}$ for $(a,b,c)\in H_3(\zz_n)$. Moreover, the $n^2$-dimensional matrix algebra $M_n$ can be embedded in $L(H_3(\zz_n))$ such that $T_t$ is invariant restricted to $M_n$. In this way, $\Ga_2-\frac{n+2}{2n}\Ga$ restricted to $M_n$ is a positive semidefinite form in $M_n$. The semigroup restricted to $M_n$ is given explicitly by
\[
T_t(v_cu_b)=e^{-t\psi(b,c)}(v_cu_b),
\]
where $\psi(b,c)=2-\de_{b,0}-\de_{c,0}$, $b,c\in\zz_n$ and
\[
v_k (e_j)=e_{j+k}, \quad u_k=\sum_{j=1}^n e^{2\pi i k(j-1)/n}\otimes e_{j,j}.
\]
Here $(e_j)$ is a basis of $\cz^n$ and $(e_{j,k})$ is the matrix unit of $M_n$. $T_t$ extends to a semigroup acting on $M_n$ because $M_n=\{v_cu_b:b,c\in\zz_n\}''$. Therefore $M_n$ satisfies \eqref{pcr3}.
\end{exam}
The previous two examples are based on the cn-length function $\psi(k)=1-\de_{k,0}$ on $\zz_n$, which gives $\Ga_2\ge\frac{n+2}{2n}\Ga$ in $L(\zz_n)$; see \cite[Section 5.3]{JZ12}. This 1-cocycle is important because it gives the number operator in the Walsh system. Another natural choice is the word length function, which is a basic notion in geometric group theory.

\begin{exam}[Word length on $\zz_n$]
Since one may embed $\zz_n$ to $\zz_{2n}$, we always assume $n$ is an even integer in this example. Consider the word length of $k\in \zz_n$ in the Cayley graph of $\zz_n$ given by $\psi(k)=\min\{k, n-k\}$. It is known that $\psi$ is conditionally negative; see \cite{JPPP}. One can also show this fact from the following explicit construction of $1$-cocycles. Let $(e_i)_{i=1}^{n/2}$ be an orthonormal basis of $\rz^{n/2}$. Define $b: \zz_n\to \rz^{n/2}$ to be
\[
b(k)=\begin{cases}
0, \quad k=0,\\
\sum_{i=1}^k e_i, \quad k=1,\cdots, n/2,\\
\sum_{i=k-n/2+1}^{n/2} e_i, \quad k=n/2+1, \cdots, n-1,
  \end{cases}
\]
and $\al: \zz_n\to O(\rz^{n/2})$ given by $\al_1(e_j)=e_{j+1}$ for $j=1,\cdots, n/2-1$ and $\al_1(e_{n/2})=-e_1$. It can be checked that $b$ is a 1-cocycle into the representation $(\al, \rz^{n/2})$ and $\psi(k)=\|b(k)\|^2$. It follows that the Gromov form $K$ is positive semidefinite. We will show that $[K(i,j)^2-K(i,j)]_{i,j=1}^{n-1}$ is a positive semidefinite matrix. By \cite[Lemma 5.3]{JZ12}, we have the following result (compare it with $\Ga_2\ge\frac{n+2}{2n}\Ga$ for the other choice of $\psi$).
\begin{prop}
  $\Ga_2\ge \Ga$ in $L(\zz_n)$.
\end{prop}
We write $K_n$ for the Gromov form of $\zz_n$. Let us take away the trivial $K_n(0,i)$'s and view $K_n$ as an $(n-1)\times (n-1)$ matrix. We need to show $K_n\bullet K_n-K_n$ is positive definite. Here $K_n\bullet K_n$ denotes the Schur product. For all even integers $2\le m\le n-2$, we write $\td{K}_m$ for the $(n-1)\times (n-1)$ matrix obtained from enlarging the size of $K_m$ by adding surrounding $0$'s so that $K_m(m/2,m/2)=\td{K}_m(n/2,n/2)$. In other words,
\begin{equation}\label{enla}
  \td{K}_m(i,j)=K_m(i-\frac{n-m}2, j-\frac{n-m}2)
\end{equation}
whenever the right-hand side is well-defined. We claim that
\begin{equation}\label{recu}
  K_n\bullet  K_n - K_n =2\sum_{\ell=1}^{n/2-1} \td{K}_{2\ell}.
\end{equation}
Since each $\td{K}_m$ is positive semidefinite, \eqref{recu} will complete the proof.

In fact, note that $K_n$ satisfies the symmetric property
$$K_n(j,i)=K_n(i,j)=K(n-j,n-i).$$
This is equivalent to saying that $K_n$ is symmetric along the two diagonals. Therefore we only need to verify \eqref{recu} entrywise in the block $B_n:=\{(i,j): 1\le i\le j\le n-i\}$. In $B_m$ for general even $m$, we have
\[
K_m(i,j)=\begin{cases}
  i, \quad (i,j)\in B_m^1:=\{(i,j): 1\le i\le j\le \frac{m}2\},\\
  \frac{m}2-j+i, \quad (i,j)\in B_m^2:=\{(i,j): 1\le i\le \frac{m}2< j\le i+\frac{m}2-1\},\\
  0, \quad (i,j)\in B_m\setminus (B_m^1\cup B_m^2).
\end{cases}
\]
By our construction, \eqref{recu} is trivial if $K_n(i,j)=0$. For $(i,j)\in B_n^1$, $\td{K}_{2\ell}$ is nonzero only if $i\ge n/2-\ell+1, j\ge n/2-\ell+1$, and for these $(i,j)$'s, $(i-\frac{n}2+\ell,j-\frac{n}2+\ell)$'s are in the block $B_{2\ell}^1$ of $K_{2\ell}$. Hence, the right-hand side of \eqref{recu} is
\[
2\sum_{\ell=n/2-i+1}^{n/2-1} (i-\frac{n}2 +\ell) = {i(i-1)}= K_n(i,j)^2-K_n(i,j).
\]
For $(i,j)\in B_n^2$, $\td{K}_{2\ell}$ is nonzero only if $j-i\le \ell-1$, and for these $(i,j)$'s, $(i-\frac{n}2+\ell,j-\frac{n}2+\ell)$'s are in the block $B_{2\ell}^2$ of $K_{2\ell}$. Then the right-hand side of \eqref{recu} is
\[
2\sum_{\ell=j-i+1}^{n/2-1} (i-j+\ell)= \Big(\frac{n}2+i-j\Big)\Big(\frac{n}2+i-j-1\Big)=K_n(i,j)^2-K_n(i,j).
\]

By the same argument as in Example \ref{exhei}, we find a new 1-cocycle on $H_3(\zz_n)$ with cn-length function $\psi(a,b,c)=|b|+|c|$ for $(a,b,c)\in H_3(\zz_n)$, where $|b|=\min\{b,n-b\}$ is the word length. The semigroup generated by this 1-cocycle satisfies the $\Ga_2$-criterion. In particular, let $T_t(v_cu_b)=e^{-t(|c|+|b|)}v_cu_b$ act on $M_n$. By the same reasoning as for the proceeding example, $T_t$ is a new semigroup acting on the matrix algebras which is subgaussian.
\end{exam}

\section*{Acknowledgements}
We thank Renming Song for helpful conversations on examples and Victor de la Pe\~na for pointing out Remark \ref{decx} to us. We are also grateful to the referee for detailed suggestions and comments which have helped to improve this paper.

\bibliographystyle{plain}
\bibliography{pcr_ref}
\end{document}